\documentclass[final,5p,times,twocolumn,numafflabel]{elsarticle}

\usepackage{amssymb}
\usepackage{amsmath}
\usepackage{amsthm}

\usepackage{comment}

\usepackage{graphicx}
\usepackage{hyperref}

\usepackage{xcolor}

\renewcommand{\d}{\mathrm{d}}
\renewcommand{\L}{\mathrm{L}}
\renewcommand{\H}{\mathrm{H}}

\newtheorem{theorem}{Theorem}

\newtheorem{definition}{Definition}

\newtheorem{remark}{Remark}
\newtheorem{proposition}{Proposition}

\journal{Elsevier}

\begin{document}

\begin{frontmatter}

\title{Riccati equations and LQ-optimal control for a class of hyperbolic PDEs}

\author[label1]{Anthony Hastir}
\author[label1]{Birgit Jacob}
\author[label2,label3]{Hans Zwart}

\affiliation[label1]{organization={School of Mathematics and Natural Sciences, University of Wuppertal},
             addressline={Gaußstraße 20},
             postcode={42119},
             city={Wuppertal},
             country={Germany}}

\affiliation[label2]{organization={Department of Applied Mathematics, University of Twente},
             addressline={P.O. Box 217},
             postcode={7500 AE},  
             city={Enschede},
             country={The Netherlands}}

\affiliation[label3]{organization={Department of Mechanical Engineering, Eindhoven University of Technology},
             addressline={P.O. Box 513},
             postcode={5600 MB},
             city={Eindhoven},
             country={The Netherlands}}

\begin{abstract}
We derive an explicit solution to the operator Riccati equation solving the Linear-Quadratic (LQ) optimal control problem for a class of boundary controlled hyperbolic partial differential equations (PDEs). Different descriptions of the system are used to obtain different representations of the operator Riccati equation. By means of an example, we illustrate the importance of considering an extended operator Riccati equation to solve the LQ-optimal control problem for our class of systems.
\end{abstract}


\end{frontmatter}

\section{Introduction}\label{sec:Intro}
The Linear-Quadratic (LQ) optimal control problem is a well-known and well-established optimization problem within control theory. Its solution can be easily described for systems driven by the continuous-time dynamics
\begin{align}
  \dot{x}(t) &= Ax(t) + Bu(t),\label{State_Finite}\\
  x(0) &= x_0,\label{Init_Finite}\\
  y(t) &= Cx(t) + Du(t),\label{Output_Finite}
\end{align}
where $A, B, C, D$ are either matrices of appropriate sizes or $A$ is an (unbounded) operator that generates a strongly continuous semigroup on the Hilbert space $X$ and $B,C,D$ are bounded linear operators, i.e., $B\in\mathcal{L}(U,X), C\in\mathcal{L}(X,Y), D\in\mathcal{L}(U,Y)$ with the Hilbert spaces $U$ and $Y$ being the input and the output spaces, respectively. In both cases, the optimal control input which minimizes the cost functional
\begin{equation}
J(x_0,u) = \int_0^\infty \Vert u(t)\Vert^2_U + \Vert y(t)\Vert^2_Y\d t
\label{Cost}
\end{equation}
over $u$ is given by the state feedback $u(t) = Kx(t) = -B^*\Theta x(t)$, where $\Theta = \Theta^*, \Theta\in\mathcal{L}(X)$, is the smallest nonnegative solution of the control algebraic Riccati equation
\begin{equation*}
A^*\Theta + \Theta A + C^*C = \left(\Theta B + C^*D\right)\left(I + D^*D\right)^{-1}\left(B^*\Theta + D^*C\right).
\end{equation*}
This Riccati equation may also be formulated in the weak sense as
\begin{align}
&\langle Ax, \Theta x\rangle + \langle \Theta x, Ax\rangle + \langle Cx, Cx\rangle\nonumber\\
&= \langle (I + D^*D)^{-1}(B^*\Theta + D^*C)x,(B^*\Theta + D^*C)x\rangle,
\label{Riccati_WeakForm}
\end{align}
for $x\in D(A)$ or $x\in X$ if $X$ is finite-dimensional. In both cases, the closed-loop system is obtained by applying the optimal control input to the original system \eqref{State_Finite}--\eqref{Output_Finite}. Moreover, the optimal cost is given by $\langle \Theta x_0,x_0\rangle = \inf_{u\in\L^2(0,\infty;U)}J(x_0,u)$. The theory of LQ-optimal control for both finite and infinite-dimensional systems with bounded input and output operators is described extensively in \cite{Curtain_Pritchard_Riccati}, \cite[Chapter 4]{CurtainPritchard} and \cite{Kwakernaak_Sivan}, or in \cite[Chapter 9]{CurtainZwart2020} for a more recent version. 

Besides a state space solution of LQ-optimal control, a frequency domain approach to LQ-optimal control can be used, known as the spectral factorization method, mainly based on the ability of the so-called Popov function to admit a spectral factor. For more details about this method, we refer the reader to \cite{CallierWinkin1990} and \cite{CallierWinkin1992}. 

The approaches described above are no more valid when the input and output operators are allowed to be unbounded. A first problem that pops up in \eqref{Riccati_WeakForm} is that $B^*\Theta x$ has no meaning so that $B^*$ has to be replaced. This problem is highlighted by a counterexample in \cite{Weiss_Zwart}. The other main problem that occurs in \eqref{Riccati_WeakForm} comes from the operator $I + D^*D$. When the original system is regular, the operator $D$ is obtained as the limit for $s$ going to $\infty$ of the transfer function $\hat{\mathbf{G}}(s)$ along the real axis. Moreover, the Popov function is defined as 
\begin{equation}
\mathbf{\Phi}(i\omega) := I + \hat{\mathbf{G}}(i\omega)^*\hat{\mathbf{G}}(i\omega), \omega\in\mathbb{R}.
\label{eq:Popov}
\end{equation}
A spectral factor of $\mathbf{\Phi}$ is a function\footnote{The Hardy space $\mathbf{H}^\infty(\mathcal{L}(U))$ consists in analytic functions $\mathbf{G}:\mathbb{C}_0^+\to\mathcal{L}(U)$ that are bounded in $\mathbb{C}_0^+$, that is, $\sup_{s\in\mathbb{C}_0^+}\Vert\mathbf{G}(s)\Vert<\infty$, where $\mathbb{C}_0^+ = \{s\in\mathbb{C}, \mathrm{Re}\,s> 0\}$.} $\boldsymbol{\chi}\in \mathbf{H}^\infty(\mathcal{L}(U))$ for which $\boldsymbol{\chi}^{-1}\in \mathbf{H}^\infty(\mathcal{L}(U))$ and that satisfies 
\begin{align}
\mathbf{\Phi}(i\omega) = \boldsymbol{\chi}(i\omega)^*\boldsymbol{\chi}(i\omega), 
\label{SpecFact_Popov}
\end{align}
for almost every $\omega\in\mathbb{R}$. This is called the spectral factorization technique. It has been developed in \cite{WeissWeiss} for weakly regular systems and further developments on this method may be found in \cite{Opmeer_Staffans} and \cite{Opmeer_MTNS}. When the spectral factor is regular, we may define $\Omega = \lim_{s\to\infty, s\in\mathbb{R}}\boldsymbol{\chi}(s)$. In the case where the input and the output operators are bounded, there holds $I + D^*D = \Omega^*\Omega$ but this is in general not the case, see e.g.~\cite{Opmeer_MTNS} for a counterexample highlighting this feature. That is why the term $I + D^*D$ has to be adapted in the more general situation. Because of the two aforementioned problems, the authors in \cite{WeissWeiss} and \cite{Staffans_LQ_1998} proposed the following operator Riccati equation
\begin{align}
  &\langle Ax, \Theta x\rangle + \langle \Theta x, Ax\rangle + \langle Cx, Cx\rangle\nonumber\\
  &= \langle (\Omega^*\Omega)^{-1}(B^*_\omega\Theta + D^*C)x,(B^*_\omega\Theta + D^*C)x\rangle,
\label{Riccati_Unbounded}
\end{align}
$x\in D(A)$, where $B^*_\omega$ is the Yosida extension of $B^*$, whose precise definition is given in e.g.~\cite{Weiss_Yosida_Ext} or \cite{Tucsnak_Weiss_LTI_Case} and will be shown later in this note. Another major difference between the bounded and the unbounded situation is that the closed-loop system is not necessarily obtained by connecting the optimal feedback to the original system, see~\cite[Sec.~3]{Weiss_Zwart}. 

Different variations of the Riccati equation \eqref{Riccati_Unbounded} exist, see e.g.~\cite{Pritchard_Salamon}, \cite{WeissWeiss}, \cite{Curtain_LQ_SIAM} and \cite{mikkola2016riccati}. 
In \cite{WeissWeiss}, the spectral factorization method is extended to unbounded control and observation operators, giving rise to a solution in the frequency domain. In \cite{Curtain_LQ_SIAM}, Riccati equations for well-posed linear systems have been considered under the assumption that $0$ is in the resolvent set of the operator dynamics, giving the possibility to define a reciprocal Riccati equation, based on the reciprocal system associated to the original one. A few years later, \cite{mikkola2016riccati} proposed generalizations of previous works on LQ-optimal control by introducing an integral Riccati equation. More recently, another type of Riccati equation called the \textit{operator node Riccati equation} has been considered in \cite{Opmeer_MTNS} and \cite{Opmeer_Staffans} for general infinite-dimensional systems. Moreover, links between the solution of the LQ-optimal control problem for general PDEs and for the discrete-time system obtained by applying the internal Cayley transform to them are presented in \cite{Opmeer_Staffans}. However, obtaining a closed-form of a solution to \eqref{Riccati_Unbounded} is in general very difficult, or even impossible. 

In this note, we consider a class of hyperbolic PDEs controlled and observed at the boundary and we propose an analytical solution to the opeartor-node Riccati equation introduced in \cite{Opmeer_Staffans}, which is based on the system node representation of our system. Moreover, we show that the obtained solution solves another version of that Riccati equation, namely \eqref{Riccati_Unbounded}. A factorization of the Popov function is also given. We take advantage of the solution to the LQ-optimal control problem for that class of systems, which has been developed in \cite{HastirJacobZwart_LQ}. The main difference between \cite{HastirJacobZwart_LQ} is as follows: in \cite{HastirJacobZwart_LQ}, we derive the solution of the LQ-optimal control problem by representing the system equivalently as an infinite-dimensional discrete-time system, which enabled us to bypass the computation of \eqref{Riccati_Unbounded}. Here, we focus on an operator Riccati equation and we solve it explicitely. We emphasize that computing the analytical solution to an operator Riccati equation like \eqref{Riccati_Unbounded} would have been challenging without taking benefit of the approach used in \cite{HastirJacobZwart_LQ}.

This note is organized as follows: the system class, its properties and the solution to the LQ-optimal control for that class are presented in Section \ref{sec:SystemClass}. An explicit solution to the operator-node Riccati equation is presented in Section \ref{sec:Riccati_SpecFact} together with a factorization of the Popov function. It is also shown that the obtained solution solves \eqref{Riccati_Unbounded}. Section \ref{sec:CE} is dedicated to an example highlighting the importance of considering \eqref{Riccati_Unbounded} to solve the LQ-optimal control problem. Some perspectives are proposed in Section \ref{sec:Ccl}.

\section{System class and linear-quadratic optimal control}\label{sec:SystemClass}

We consider partial differential equations (PDEs) of the form 
\begin{align}
\frac{\partial z}{\partial t}(\zeta,t) &= -\frac{\partial}{\partial \zeta}(\lambda_0(\zeta)z(\zeta,t)),\,\,\,\,\,\,z(\zeta,0) = z_0(\zeta)\in X,\label{StateEquation_Diag_Uniform}
\end{align}
where $\lambda_0(\cdot)\geq \epsilon>0$ is a bounded measurable function and $X := \L^2(0,1;\mathbb{C}^n)$ is the Hilbert state space equipped with the inner product
\begin{equation*}
\langle f,g\rangle_X := \int_0^1g^*(\zeta)\lambda_0(\zeta)f(\zeta)\d\zeta.
\end{equation*}
The parameters $t\geq 0$ and $\zeta\in [0,1]$ stand for time and space, respectively. To the PDEs \eqref{StateEquation_Diag_Uniform}, the following boundary inputs and outputs are considered
\begin{align}
\left[\begin{matrix}0\\ I\end{matrix}\right]u(t) &= -K\lambda_0(0)z(0,t) - L\lambda_0(1)z(1,t),\label{Input_Diag_Uniform}\\
y(t) &= -K_y\lambda_0(0)z(0,t) - L_y\lambda_0(1)z(1,t),\label{Output_Diag_Uniform}
\end{align}
with $u(t)\in\mathbb{C}^p =: U$ and $y(t)\in\mathbb{C}^m =: Y$ being the inputs and the outputs, respectively. The matrices $K, L, K_y$ and $L_y$ satisfy $K\in\mathbb{C}^{n\times n}, L\in\mathbb{C}^{n\times n}, K_y\in\mathbb{C}^{m\times n}$ and $L_y\in\mathbb{C}^{m\times n}$, respectively. It is also possible to add the term $M(\zeta)z(\zeta,t)$ in the PDEs \eqref{StateEquation_Diag_Uniform} where $M(\cdot)\in\L^\infty(0,1;\mathbb{C}^{n\times n})$. In that case, it is shown in \cite{HastirJacobZwart_LCSS} and \cite{HastirJacobZwart_LQ} that the change of variables $\tilde{z}(\zeta,t) = Q(\zeta)z(\zeta,t)$ makes the term $M(\zeta)z(\zeta,t)$ disappear\footnote{The matrix $Q$ is the solution to the matrix differential equation $Q'(\zeta) = -\lambda_0(\zeta)^{-1}Q(\zeta)M(\zeta),\,\,\, Q(0) = I$.}. It only affects the boundary conditions slightly but they remain of the same form. For this reason, we will not add the term $M(\zeta)z(\zeta,t)$ in the PDEs \eqref{StateEquation_Diag_Uniform}. This is without loss of generality.

The class \eqref{StateEquation_Diag_Uniform}--\eqref{Output_Diag_Uniform} is able to describe many different physical phenomena. For instance, the dynamics of networks of electric lines, chains of density-velocity systems or genetic regulatory networks may be written with hyperbolic PDEs similar to \eqref{StateEquation_Diag_Uniform}--\eqref{Output_Diag_Uniform}. In a more general sense, linear conservation laws may be modeled as \eqref{StateEquation_Diag_Uniform}--\eqref{Output_Diag_Uniform}. We refer the reader to e.g.~\cite{BastinCoron_Book} for a larger set of applications whose dynamics admit the representation \eqref{StateEquation_Diag_Uniform}--\eqref{Output_Diag_Uniform}. A more general version of that class is presented and analyzed in the book \cite[Chapter 6]{Luo_Guo}, allowing for instance for a diagonal matrix with different entries instead of the function $\lambda_0$. The class \eqref{StateEquation_Diag_Uniform}--\eqref{Output_Diag_Uniform} has also been considered in several works. Among them, Lyapunov exponential stability for \eqref{StateEquation_Diag_Uniform}--\eqref{Output_Diag_Uniform} has been studied in \cite{Bastin_2007, Coron_2007_Lyapunov, Diagne_2012}, whereas backstepping for boundary control and output feedback stabilization for networks of linear one-dimensional hyperbolic PDEs has been considered in \cite{Auriol_2020, Auriol_BreschPietri, Vazquez_2011_Backstepping}. Boundary feedback control for \eqref{StateEquation_Diag_Uniform}--\eqref{Output_Diag_Uniform} has been considered in \cite{DeHalleux_2003,Prieur_2008_ConservationLaws,Prieur_Winkin_2018}. Controllability and finite-time boundary control have also been studied in e.g.~\cite{Chitour_2023, Coron_2021_NullContr} and \cite{Auriol_2016,Coron_2021_Stab_FiniteTime}. Boundary output feedback stabilization for \eqref{StateEquation_Diag_Uniform}--\eqref{Output_Diag_Uniform} has been investigated in \cite{Tanwani2018} in the case of measurement errors, leading to the study of input-to-state stability for those systems. More recently, Riesz-spectral property of the homogeneous part of \eqref{StateEquation_Diag_Uniform}--\eqref{Output_Diag_Uniform} has been analyzed in \cite{HastirJacobZwart_LCSS}. Moreover, linear-quadratic optimal control has been developed for \eqref{StateEquation_Diag_Uniform}--\eqref{Output_Diag_Uniform} in \cite{HastirJacobZwart_LQ}. In addition, $H^\infty$-controllers have been computed for \eqref{StateEquation_Diag_Uniform}--\eqref{Output_Diag_Uniform} in \cite{HastirJacobZwart_Hinf}.

\subsection{System theoretic properties}

Here, we emphasize some properties of the class \eqref{StateEquation_Diag_Uniform}--\eqref{Output_Diag_Uniform} that will be useful later in this manuscript. We start by well-posedness. It has been shown in \cite{HastirJacobZwart_LCSS} that the well-posedness of the boundary control system \eqref{StateEquation_Diag_Uniform}--\eqref{Output_Diag_Uniform} is equivalent to the invertibility of the matrix $K$. We make this assumption in what follows. Our next objective is to write \eqref{StateEquation_Diag_Uniform}--\eqref{Output_Diag_Uniform} as an abstract controlled and observed Cauchy problem of the form
\begin{equation}
\left\{
\begin{array}{l}
\dot{z}(t) = Az(t) + Bu(t)\\
y(t) = Cz(t) + Du(t),
\end{array}\right.
\end{equation}
where $A:D(A)\subset X\to X$ is an unbounded linear operator, $B:U\to X_{-1}$ and $C:D(A)\to Y$ are the control\footnote{The space $X_{-1}$ is the completion of $X$ with respect to the norm $\Vert\cdot\Vert_{-1} := \Vert (\beta I-A)^{-1}\cdot\Vert$ for some $\beta\in\rho(A)$, the resolvent set of $A$.} and the observation operators and $D:U\to Y$ is the feedthrough operator. First observe that by defining the matrices 
\begin{equation}
\begin{array}{ll}
A_d := -K^{-1}L, & B_d := -K^{-1}\left[\begin{matrix}0\\ I\end{matrix}\right],\\
C_d := (K_yK^{-1}L-L_y), & D_d := K_yK^{-1}\left[\begin{matrix}0\\ I\end{matrix}\right],
\end{array}\label{Matrices_Discrete}
\end{equation}
the input and the output equations \eqref{Input_Diag_Uniform}--\eqref{Output_Diag_Uniform} may be rewritten as
\begin{align}
B_d u(t) &= \lambda_0(0) z(0,t) - A_d \lambda_0(1)z(1,t),\label{Input_New}\\
y(t) &= C_d \lambda_0(1)z(1,t) + D_d u(t),\label{Output_New}
\end{align}
respectively. According to \eqref{StateEquation_Diag_Uniform} and \eqref{Input_New}, the operator $A$ is given by 
\begin{align}
Az &= -\frac{\d}{\d\zeta}(\lambda_0 z)\label{Op(A)}\\
D(A) &= \{z\in X, \lambda_0 z\in\H^1(0,1;\mathbb{C}^n),\nonumber\\
    &\hspace{1cm} (\lambda_0z)(0) - A_d(\lambda_0z)(1) = 0\}.\label{D(A)}
\end{align}
As is motivated in \cite{TucsnakWeiss}, it is in general more common to define the adjoint of the operator $B$ instead of $B$ itself. The adjoint of $B$, denoted by $B^*$, is defined on $D(A^*)$ with values in $U$, where $A^*$ is the adjoint of $A$. Before computing the operator $B^*$, observe that $A^*$ is defined as 
\begin{align*}
A^* z &= \frac{\d}{\d\zeta}(\lambda_0 z),\\
D(A^*) &= \{z\in X, \lambda_0 z\in\H^1(0,1;\mathbb{C}^n),\nonumber\\
    &\hspace{1cm} (\lambda_0z)(1) - A_d^*(\lambda_0z)(0) = 0\},
\end{align*}
see~\cite[Prop.~3.4.3]{Augner_PhD}.
In order to compute the adjoint operator of $B$, we use the following relation 
\begin{equation}
\langle Lz,\psi\rangle = \langle z,A^*\psi\rangle + \langle Gz,B^*\psi\rangle,
\label{TucWei_AdjointB}
\end{equation}
for $\psi\in D(A^*)$ and $z\in D(L)$, see \cite[Chapter 10]{TucsnakWeiss}. In the above expression, the operator $L$ is defined as
\begin{align*}
Lz &= -\frac{\d}{\d\zeta}(\lambda_0 z)\\
D(L) &= \{z\in X, \lambda_0z\in\H^1(0,1;\mathbb{C}^n),\\
&\hspace{2cm} 0 = [\begin{smallmatrix}I & 0\end{smallmatrix}]K\left[(\lambda_0z)(0) - A_d(\lambda_0z)(1)\right]\},
\end{align*}
which is the extension of $A$ to $D(L)$. The role of the matrix $[\begin{smallmatrix}I & 0\end{smallmatrix}]\in\mathbb{R}^{(n-p)\times n}$ in $D(L)$ is to keep the uncontrolled boundary conditions. The control inputs are placed in the operator $G$, given by $u(t) = Gz(t)$, that is,
\begin{equation}
Gz = -\left[\begin{smallmatrix}0 & I\end{smallmatrix}\right]K\left[(\lambda_0z)(0) - A_d(\lambda_0z)(1)\right], 
\label{OpG}
\end{equation}
for $z\in D(L)$, where \eqref{Input_New} has been used. Expanding \eqref{TucWei_AdjointB} for $\psi\in D(A^*)$ and $z\in D(L)$ yields that
\begin{align*}
(B^*\psi)^*\cdot(Gz) = -\int_0^1 \lambda_0\psi^*\frac{\d}{\d\zeta}(\lambda_0z)\d\zeta - \langle z,A^*\psi\rangle.
\end{align*}
An integration by parts has the consequence that
\begin{align*}
(B^*\psi)^*\cdot(Gz) = (\lambda_0\psi)(0)^*(\lambda_0z)(0) - (\lambda_0\psi)(1)^*(\lambda_0z)(1).
\end{align*}
By using the boundary conditions induced by $D(A^*)$, the previous equality may be written equivalently as
\begin{align*}
(B^*\psi)^*\cdot(Gz) = (\lambda_0\psi)(0)^*\left[(\lambda_0z)(0) - A_d(\lambda_0z)(1)\right].
\end{align*}
Going a step further implies that
\begin{align*}
&(B^*\psi)^*\cdot(Gz)\\
&= (\lambda_0\psi)(0)^*K^{-1}K\left[(\lambda_0z)(0) - A_d(\lambda_0z)(1)\right]\\
&=(\lambda_0\psi)(0)^*K^{-1}\left[\begin{matrix}[\begin{smallmatrix}I & 0\end{smallmatrix}]K\left[(\lambda_0z)(0) - A_d(\lambda_0z)(1)\right]\\
[\begin{smallmatrix}0 & I\end{smallmatrix}]K\left[(\lambda_0z)(0) - A_d(\lambda_0z)(1)\right]
\end{matrix}\right],
\end{align*}
where the matrices $[\begin{smallmatrix}I & 0\end{smallmatrix}]$ and $[\begin{smallmatrix}0 & I\end{smallmatrix}]$ in the last line are such that $[\begin{smallmatrix}I & 0\end{smallmatrix}]\in\mathbb{R}^{(n-p)\times n}$ and $[\begin{smallmatrix}0 & I\end{smallmatrix}]\in\mathbb{R}^{p\times n}$. Considering that $z\in D(L)$, there holds
\begin{align}
&(B^*\psi)^*\cdot(Gz)\nonumber\\
&=(\lambda_0\psi)(0)^*K^{-1}\left[\begin{matrix}0_{n-p}\\
[\begin{smallmatrix}0 & I\end{smallmatrix}]K\left[(\lambda_0z)(0) - A_d(\lambda_0z)(1)\right]
\end{matrix}\right],\label{IP}
\end{align}
where the notation $0_{n-p}$ stands for a zero column vector of size $n-p$. Since only the last $p$ components of the row vector $(\lambda_0\psi)(0)^*K^{-1}$ play a role in the inner product \eqref{IP}, we have that
\begin{align*}
&(B^*\psi)^*\cdot(Gz)\nonumber\\
&= (\lambda_0\psi)(0)^*K^{-1}\left[\begin{smallmatrix}0\\ I\end{smallmatrix}\right]\left[\begin{smallmatrix}0 & I\end{smallmatrix}\right]K\left[(\lambda_0z)(0) - A_d(\lambda_0z)(1)\right]\\
&= (\lambda_0\psi)(0)^*B_d\cdot Gz
\end{align*}
Since the operator $G$ is onto as an operator from $D(L)$ into $U$, we have that the equality 
\begin{align*}
(B^*\psi)^*\cdot u = (\lambda_0\psi)(0)^*B_d\cdot u
\end{align*}
has to hold for all $u\in U$. This has the consequence that 
\begin{equation}
B^*:D(A^*)\to\mathbb{C}^p,~ B^*\psi = B_d^*(\lambda_0\psi)(0).
\label{Adj_B}
\end{equation}
Note that with the boundary condition in $D(A^*)$ we could as well have it as an evaluation at $\zeta=1$.

With a perspective of solving the operator Riccati equation \eqref{Riccati_Unbounded}, we give also the expression for the operator $B_\omega^*$. In order to compute $B_\omega^*$, let us first have a look at the resolvent operator associated to $A^*$. For $g\in X$, it can be shown that
\begin{align*}
    &((sI-A^*)^{-1}g)(\zeta)\\
    = &~\frac{1}{\lambda_0(\zeta)}e^{-s p(1-\zeta)}A_d^*(I-e^{-sp(1)}A_d^*)^{-1}\int_0^1e^{-sp(\eta)}g(\eta)\d\eta\\
    &+\frac{1}{\lambda_0(\zeta)}\int_\zeta^1 e^{-sp(\eta-\zeta)}g(\eta)\d\eta,
\end{align*}
where the function $p:[0,1]\to\mathbb{R}^+$ is defined by 
\begin{equation*}
p(\zeta) := \int_0^\zeta\lambda_0(\eta)^{-1}\d\eta.
\end{equation*}

Now consider a function $g\in X$ that is such that $\lambda_0g$ admits a right limit at $0$. Applying the operator $B^*$ to $s(sI-A^*)^{-1}g$ and taking the limit for $s$ going to $\infty$ gives
\begin{align*}
    &\lim_{s\to\infty, s\in\mathbb{R}} B^*s(sI-A^*)^{-1}g\\
    &= \lim_{s\to\infty, s\in\mathbb{R}} B_d^* s(I-e^{-sp(1)}A_d^*)^{-1}\int_0^1 e^{-sp(\eta)}g(\eta)\d\eta\\
    &=\lim_{s\to\infty, s\in\mathbb{R}} B_d^* s\int_0^{p(1)} e^{-s\sigma}\lambda_0(p^{-1}(\sigma))g(p^{-1}(\sigma))\d\sigma,
\end{align*}
where $p^{-1}$ is the reciprocal function of $p$. In particular, we have that $p^{-1}(0) = 0$. As a consequence, by using the initial value theorem of the Laplace transform, see~\cite[Chap.~1]{Spiegel65}, there holds 
\begin{equation}
B_\omega^* g = B_d^*(\lambda_0g)(0^+).
\label{Yosida_Ext_B}
\end{equation}
Hence the domain of $B_\omega^*$ contains those $g$ for which $\lambda_0 g$ admits a right limit at $0$.

It remains to compute the operators $C$ and $D$. We start with the operator $D$. For this, observe that the transfer function of \eqref{StateEquation_Diag_Uniform}--\eqref{Output_Diag_Uniform} is given by 
\begin{equation}
\hat{\mathbf{G}}(s) = C_d(e^{sp(1)}-A_d)^{-1}B_d + D_d,
\label{eq:TrF}
\end{equation}
see \cite{HastirJacobZwart_Hinf}. According to \cite{WeissTRF}, the operator $D$ is defined as $Du = \lim_{s\in\mathbb{R},s\to\infty}\hat{\mathbf{G}}(s)$ provided that the limit exists, that is, provided that \eqref{StateEquation_Diag_Uniform}--\eqref{Output_Diag_Uniform} is a regular system. The regularity of the system is guaranteed thanks to \cite[Theorem 13.2.2]{JacobZwart}. Now observe that 
\begin{align*}
\lim_{s\in\mathbb{R},s\to\infty} \Vert \hat{\mathbf{G}}(s) - D_d\Vert = 0.
\end{align*}
As a consequence, the operator $D$ is expressed as 
\begin{equation}
D : U\to Y,~ Du = D_d u.
\label{OpD}
\end{equation}
Looking at \eqref{Output_New}, it is then natural to define the operator $C$ as
\begin{equation}
C: D(A)\to Y,~ Cz = C_d(\lambda_0z)(1).
\label{OpC}
\end{equation}

\subsection{LQ-optimal control}

We review some results presented in \cite{HastirJacobZwart_LQ} on the LQ-optimal control problem for \eqref{StateEquation_Diag_Uniform}--\eqref{Output_Diag_Uniform}. According to \cite{HastirJacobZwart_LQ}, the optimal control input that minimizes the cost \eqref{Cost} is given by the boundary feedback operator
\begin{equation}
u^{opt}(t) = F_d\lambda_0(1)z^{opt}(1,t),
\label{eq:OptimalFeedback}
\end{equation}
with $F_d := -P^{-1}V$, where $P := I + D_d^*D_d + B_d^*\Pi B_d$ and $V := D_d^*C_d + B_d^*\Pi A_d$, in which $\Pi$ is the smallest nonnegative solution of the following control algebraic Riccati equation (CARE)
\begin{equation}
A_d^*\Pi A_d - \Pi + C_d^*C_d = V^*P^{-1}V,
\label{CARE}
\end{equation}
and the matrices $A_d, B_d, C_d$ and $D_d$ are given \eqref{Matrices_Discrete}. Note that the solution $\Pi$ needs not to be unique. The uniqueness of $\Pi$ is guaranteed if the filter algebraic Riccati equation (FARE)
\begin{equation}
A_d\tilde{\Pi} A_d^* - \tilde{\Pi} + B_d^*B_d = WT^{-1}W^*,
\label{FARE}
\end{equation}
where $W := B_dD_d^* + A_d\tilde{\Pi} C_d^*$ and $T := I + D_dD_d^* + C_d\tilde{\Pi}C_d^*$, has a nonnegative self-adjoint solution and if that the matrix $A_\Pi := A_d + B_dF_d$ is stable in the sense $r(A_\Pi)<1$, $r$ being the spectral radius. In this case, the optimal cost is given by 
\begin{equation}
J(z_0,u^{opt}) = \langle z_0,\Pi z_0\rangle_X.
\label{OptimalCost}
\end{equation}

\section{Riccati equations and spectral factorization}\label{sec:Riccati_SpecFact}
Our aim in this section is to look for an infinite-dimensional operator Riccati equation that would define the optimal control operator given in \eqref{eq:OptimalFeedback}. 

Note that the notations $P := I + D_d^*D_d + B_d^*\Pi B_d$ and $V := D_d^*C_d + B_d^*\Pi A_d$ are used all along this section.

\subsection{Operator node Riccati equation}

The idea of using the concept of system nodes in Riccati equations appeared in \cite{Opmeer_Staffans} and was used again in \cite{Opmeer_MTNS}. We first recall the notion of a system node, see \cite[Definition 1.1.1]{Sta05book}.

\begin{definition}\label{def:SystemNode}
Let $X,U,Y$ be Banach spaces. An operator $S: D(S)\subset \left[\begin{smallmatrix}X\\ U\end{smallmatrix}\right]\to \left[\begin{smallmatrix}X\\ Y\end{smallmatrix}\right]$ is called a system node if 
\begin{itemize}
\item $S$ is a closed operator;
\item if we split $S = \left[\begin{smallmatrix}S_X\\ S_Y\end{smallmatrix}\right]$ in accordance with the splitting of the range space $\left[\begin{smallmatrix}X\\ Y\end{smallmatrix}\right]$, then $S_X$ with domain $D(S)$ is closed;
\item the operator $A$ defined by $Ax := S_X\left[\begin{smallmatrix}x\\ 0\end{smallmatrix}\right]$ with domain $D(A) = \{x\in X, \left[\begin{smallmatrix}x\\ 0\end{smallmatrix}\right]\in D(S)\}$ is the generator of a strongly continuous semigroup on $X$;
\item for every $u\in U$, there exists $x$ such that $\left[\begin{smallmatrix}x\\ u\end{smallmatrix}\right]\in D(S)$.
\end{itemize}
\end{definition}
We use the notations $S_X =: A\& B$ and $S_Y =: C\& D$ in what follows. 
By considering \eqref{Input_New} and \eqref{Output_New}, the system node associated to \eqref{StateEquation_Diag_Uniform}--\eqref{Output_Diag_Uniform} is given by 
\begin{align}
S\left[\begin{matrix}z\\ u\end{matrix}\right] &= \left[\begin{matrix}A\& B\\ C\& D\end{matrix}\right]\left[\begin{matrix}z\\ u\end{matrix}\right] = \left[\begin{matrix}-\frac{\d}{\d\zeta}(\lambda_0 z)\\ C_d\lambda_0(1)z(1) + D_d u\end{matrix}\right]\label{Op(S)}\\
D(S) &= \left\{\left[\begin{matrix}z\\ u\end{matrix}\right]\in\left[\begin{matrix}X\\ \mathbb{C}^p\end{matrix}\right], \frac{\d}{\d\zeta}(\lambda_0 z)\in X,\right.\nonumber\\
&\hspace{1cm}\left.\lambda_0(0)z(0) - A_d\lambda_0(1)z(1) = B_d u\right\}\label{D(S)}.
\end{align}

It can be shown that \eqref{Op(S)}--\eqref{D(S)} defines a system node as described in Definition \ref{def:SystemNode}. The proof is omitted here.

The notion of operator node Riccati equation is given in the following definition, see \cite{Opmeer_MTNS}, \cite{Opmeer_Staffans}.

\begin{definition}
Let $S$ be a system node and $R = R^*, R>0$. The operators $Z = Z^*\in\mathcal{L}(X)$ and $K\& L: D(S)\to U$ are called a solution of the operator node Riccati equation for $S$ if
\begin{align}
&\left\langle A\& B\left[\begin{matrix}z\\ u\end{matrix}\right], Zz\right\rangle + \left\langle Zz, A\& B\left[\begin{matrix}z\\ u\end{matrix}\right]\right\rangle + \left\Vert C\& D\left[\begin{matrix}z\\ u\end{matrix}\right]\right\Vert^2 + \left\langle Ru, u\right\rangle\nonumber\\
&\hspace{5cm} = \left\Vert K\& L\left[\begin{matrix}z\\ u\end{matrix}\right]\right\Vert^2
\label{ON_Riccati}
\end{align}
for all $\left[\begin{smallmatrix}z\\ u\end{smallmatrix}\right]\in D(S)$. The operator $Z$ satisfies
\begin{align*}
\langle Zz_0,z_0\rangle &= \inf_{u\in\L^2(0,\infty;\mathbb{C}^p)} \int_0^\infty \langle Ru(t),u(t)\rangle_U + \Vert y(t)\Vert^2_Y\d t,
\end{align*}
where $z_0$ is the initial condition.
\end{definition}

The solution of the operator node Riccati equation \eqref{ON_Riccati} is given in the following theorem.

\begin{theorem}
Let $S$ be the system node defined in \eqref{Op(S)}--\eqref{D(S)}, $\Pi$ be the nonnegative solution of the CARE \eqref{CARE} and $R = I_U$. Then the operators $Z = \Pi I_X$ and $K\& L: D(S)\to U$,
\begin{equation*}
K\& L\left[\begin{smallmatrix}z\\ u\end{smallmatrix}\right] = P^{-\frac{1}{2}}V\lambda_0(1)z(1) + P^{\frac{1}{2}} u,
\end{equation*}
are the solution of the operator node Riccati equation \eqref{ON_Riccati}.
\end{theorem}

\begin{proof}
Take $\left[\begin{smallmatrix}z\\ u\end{smallmatrix}\right]\in D(S)$. There holds
\begin{align*}
&\left\langle A\& B\left[\begin{matrix}z\\ u\end{matrix}\right], Z z\right\rangle + \left\langle Z z, A\& B\left[\begin{matrix}z\\ u\end{matrix}\right]\right\rangle\\
&= - \int_0^1 z^*\Pi\lambda_0\frac{\d}{\d\zeta}\left(\lambda_0 z\right)\d\zeta - \int_0^1 \frac{\d}{\d\zeta}\left(\lambda_0 z\right)^*\Pi\lambda_0 z\d\zeta.
\end{align*}
An integration by parts implies that
\begin{align*}
&\left\langle A\& B\left[\begin{matrix}z\\ u\end{matrix}\right], Z z\right\rangle + \left\langle Z z, A\& B\left[\begin{matrix}z\\ u\end{matrix}\right]\right\rangle\\
&= -(\lambda_0z)(1)^*\Pi(\lambda_0z)(1) + (\lambda_0z)(0)^*\Pi(\lambda_0z)(0).
\end{align*}
By using \eqref{D(S)}, we get
\begin{align*}
&\left\langle A\& B\left[\begin{matrix}z\\ u\end{matrix}\right], Zz\right\rangle + \left\langle Zz, A\& B\left[\begin{matrix}z\\ u\end{matrix}\right]\right\rangle\\
= &~-(\lambda_0z)(1)^*\Pi (\lambda_0z)(1)\\
& + (B_du + A_d(\lambda_0z)(1))^*\Pi (B_du + A_d(\lambda_0z)(1))\\
= &~(\lambda_0z)(1)^*\left[A_d^*\Pi A_d-\Pi \right](\lambda_0z)(1) + u^*B_d^*\Pi B_du\\
&+ u^*B_d^*\Pi A_d(\lambda_0z)(1) + (\lambda_0z)(1)^*A_d^*\Pi B_du.
\end{align*}
In addition, we have that
\begin{align*}
&\left\Vert C\& D\left[\begin{matrix}z\\ u\end{matrix}\right]\right\Vert^2 + \Vert u\Vert^2\\
= &~\left(C_d(\lambda_0z)(1) + D_du\right)^*\left(C_d(\lambda_0z)(1) + D_du\right) + u^*u\\
= &~(\lambda_0z)(1)^*C_d^*C_d(\lambda_0z)(1) + (\lambda_0z)(1)^*C_d^*D_du\\
&+ u^*D_d^*C_d(\lambda_0z)(1) + u^*\left[D_d^*D_d+I\right]u.
\end{align*}
As a consequence, there holds
\begin{align*}
&\left\langle A\& B\left[\begin{matrix}z\\ u\end{matrix}\right], Zz\right\rangle + \left\langle Zz, A\& B\left[\begin{matrix}z\\ u\end{matrix}\right]\right\rangle + \left\Vert C\& D\left[\begin{matrix}z\\ u\end{matrix}\right]\right\Vert^2 + \Vert u\Vert^2\\
= &(\lambda_0z)(1)^*\left[A_d^*\Pi A_d-\Pi  + C_d^*C_d\right](\lambda_0z)(1)\\
&+ u^*Pu + (\lambda_0z)(1)^*V^*u + u^*V(\lambda_0z)(1).
\end{align*}
Using the CARE \eqref{CARE} implies that
\begin{align*}
&\left\langle A\& B\left[\begin{matrix}z\\ u\end{matrix}\right], Zz\right\rangle + \left\langle Zz, A\& B\left[\begin{matrix}z\\ u\end{matrix}\right]\right\rangle + \left\Vert C\& D\left[\begin{matrix}z\\ u\end{matrix}\right]\right\Vert^2 + \Vert u\Vert^2\\
= &~(\lambda_0z)(1)^*\left[V^*P^{-1}V\right](\lambda_0z)(1)\\
&+ u^*Pu+ (\lambda_0z)(1)^*V^*u + u^*V(\lambda_0z)(1)\\
= &~\left\Vert P^{-\frac{1}{2}}V(\lambda_0z)(1) + P^\frac{1}{2}u\right\Vert^2\\
= &~\left\Vert K \& L\left[\begin{matrix}z\\ u\end{matrix}\right]\right\Vert^2.\qedhere
\end{align*}
\end{proof}
As is mentioned in \cite[Section 5]{Opmeer_Staffans}, the optimal control $u^{opt}$ is the solution of $K \& L\left[\begin{smallmatrix}z^{opt}\\u^{opt}\end{smallmatrix}\right] = 0$. This is equivalent to 
\begin{align*}
u^{opt} &= -P^{-1}V(\lambda_0z^{opt})(1) = F_d(\lambda_0z^{opt})(1),
\end{align*}
which coincides with \eqref{eq:OptimalFeedback}.

\subsection{\textit{Weiss--Weiss} operator Riccati equation and spectral factorization}

We start this section by giving a function $\boldsymbol{\chi}$ that formally satisfies \eqref{SpecFact_Popov}. However, without adding additional assumption, the obtained $\boldsymbol{\chi}$ may not be called a spectral factor because the conditions $\boldsymbol{\chi}\in\mathbf{H}^\infty(\mathcal{L}(U))$ and $\boldsymbol{\chi}^{-1}\in\mathbf{H}^\infty(\mathcal{L}(U))$ are not necessarily satisfied. 

The following proposition gives a candidate for a spectral factor of the Popov function associated to the system node $S$ defined in \eqref{Op(S)}--\eqref{D(S)}. 

\begin{proposition}
The function $\boldsymbol{\chi}$ defined by
\begin{align*}
\boldsymbol{\chi}(s) = P^\frac{1}{2}\left[I - F_d(e^{s p(1)}-A_d)^{-1}B_d\right]
\end{align*}
satisfies formally $\boldsymbol{\chi}(-\overline{s})^*\boldsymbol{\chi}(s) = I + \hat{\mathbf{G}}(-\overline{s})^*\hat{\mathbf{G}}(s)$, where $\hat{\mathbf{G}}(s)$ is the transfer function associated to \eqref{StateEquation_Diag_Uniform}--\eqref{Output_Diag_Uniform}, whose expression is given in \eqref{eq:TrF}.
\end{proposition}

\begin{proof}
First observe that the adjoint of $\chi$ is given by
\begin{align*}
    \boldsymbol{\chi}(s)^* = \left[I - B_d^*(e^{\overline{s}p(1)}-A_d^*)^{-1}F_d^*\right]P^\frac{1}{2}.
\end{align*}
This implies that
\begin{align*}
    &\boldsymbol{\chi}(-\overline{s})^*\boldsymbol{\chi}(s)\\
    = &\left[I - B_d^*(e^{-sp(1)}-A_d^*)^{-1}F_d^*\right]P\left[I - F_d(e^{sp(1)}-A_d)^{-1}B_d\right]\\
    = &~P + B_d^*\left(e^{-sp(1)}-A_d^*\right)^{-1}V^* + V\left(e^{sp(1)}-A_d\right)^{-1}B_d\\
    &+ B_d^*\left(e^{-sp(1)}-A_d^*\right)^{-1}V^*P^{-1}V\left(e^{sp(1)}-A_d\right)^{-1}B_d,
\end{align*}
in which the equality $F_d = -P^{-1}V$, see \eqref{eq:OptimalFeedback}, has been used. Thanks to \eqref{CARE}, the previous equality may be re written as
\begin{align*}
    &\boldsymbol{\chi}(-\overline{s})^*\boldsymbol{\chi}(s)\\
    = &~I + D_d^*D_d + B_d^*\Pi B_d + B_d^*\left(e^{-sp(1)}-A_d^*\right)^{-1}C_d^*D_d\\
    &+ B_d^*\left(e^{-sp(1)}-A_d^*\right)^{-1}A_d^*\Pi B_d + D_d^*C_d\left(e^{sp(1)}-A_d\right)^{-1}B_d\\
    &+ B_d^*\Pi A_d\left(e^{sp(1)}-A_d\right)^{-1}B_d\\
    &+B_d^*\left(e^{-sp(1)}-A_d^*\right)^{-1}A_d^*\Pi A_d\left(e^{sp(1)}-A_d\right)^{-1}B_d\\
    &- B_d^*\left(e^{-sp(1)}-A_d^*\right)^{-1}\Pi \left(e^{sp(1)}-A_d\right)^{-1}B_d\\
    &+B_d^*\left(e^{-sp(1)}-A_d^*\right)^{-1}C_d^*C_d\left(e^{sp(1)}-A_d\right)^{-1}B_d.
\end{align*}
According to the expression of the transfer function \eqref{eq:TrF}, there holds
\begin{align*}
    &\boldsymbol{\chi}(-\overline{s})^*\boldsymbol{\chi}(s)\\
    =&~ I + \hat{\mathbf{G}}(-\overline{s})^*\hat{\mathbf{G}}(s)\\
    & + B_d^*\left[\Pi  + \left(e^{-sp(1)}-A_d^*\right)^{-1}A_d^*\Pi + \Pi A_d\left(e^{sp(1)}-A_d\right)^{-1}\right.\\
    & + \left(e^{-sp(1)}-A_d^*\right)^{-1}A_d^*\Pi A_d\left(e^{sp(1)}-A_d\right)^{-1}\\
    &\left.- \left(e^{-sp(1)}-A_d^*\right)^{-1}\Pi \left(e^{sp(1)}-A_d\right)^{-1}\right]B_d.
\end{align*}
Observe now that 
\begin{align*}
&\Pi  + \left(e^{-sp(1)}-A_d^*\right)^{-1}A_d^*\Pi + \Pi A_d\left(e^{sp(1)}-A_d\right)^{-1}\\
&+ \left(e^{-sp(1)}-A_d^*\right)^{-1}A_d^*\Pi A_d\left(e^{sp(1)}-A_d\right)^{-1}\\
&- \left(e^{-sp(1)}-A_d^*\right)^{-1}\Pi \left(e^{sp(1)}-A_d\right)^{-1}\\
=&~\left(e^{-sp(1)}-A_d^*\right)^{-1}\left[\left(e^{-sp(1)}-A_d^*\right)\Pi\left(e^{sp(1)}-A_d\right)\right.\\
&+ A_d^*\Pi\left(e^{sp(1)}-A_d\right) + \left(e^{-sp(1)}-A_d^*\right)\Pi A_d\\
&\left.+ A_d^*\Pi A_d - \Pi\right]\left(e^{sp(1)}-A_d\right)^{-1}\\
=&~0,
\end{align*}
which concludes the proof.
\end{proof}

\begin{remark}\label{rem:SpecFact}
To call $\boldsymbol{\chi}$ a spectral factor, it should satisfy $\boldsymbol{\chi}\in\mathbf{H}^\infty(\mathcal{L}(U))$ and $\boldsymbol{\chi}^{-1}\in\mathbf{H}^\infty(\mathcal{L}(U))$ besides the equality \eqref{SpecFact_Popov}. With our assumptions, this is not guaranteed a priori. An additional condition that is needed to call $\boldsymbol{\chi}$ a spectral factor is the coercivity of the Popov function $\boldsymbol{\Phi}(i\omega)$. It is easy to see that for all $u\in U$ and all $\omega\in\mathbb{R}$, there holds
\begin{align*}
\langle u,\boldsymbol{\Phi}(i\omega)u\rangle_U &= \langle u,(I + \hat{\mathbf{G}}(i\omega)^*\hat{\mathbf{G}}(i\omega))u\rangle_U\\
&=\Vert u\Vert^2_U + \Vert \hat{\mathbf{G}}(i\omega)u\Vert^2\\
&\geq \Vert u\Vert_U^2,
\end{align*}
which shows that $\mathbf{\Phi}(i\omega)$ is coercive. What will be important to define another form of the Riccati equation needed for the LQ-optimal control problem is the coercivity of $\boldsymbol{\Phi}$ and the regularity of $\boldsymbol{\chi}$ viewed as a transfer function of a system node. By noting that
\begin{align*}
\boldsymbol{\chi}(s) &= P^\frac{1}{2} + P^{-\frac{1}{2}}V\left(e^{sp(1)}-A_d\right)^{-1}B_d,
\end{align*}
it can be seen that the $\boldsymbol{\chi}$ is the transfer function of the system node $S_{sp} := \left[\begin{smallmatrix}A\& B\\ K \& L\end{smallmatrix}\right], D(S_{sp}) = D(S)$. Moreover, both the system nodes $S$ and $S_{sp}$ are regular, with $\lim_{s\to\infty, s\in\mathbb{R}} \hat{\mathbf{G}}(s) = D_d$ and $\lim_{s\to\infty, s\in\mathbb{R}} \boldsymbol{\chi}(s) = P^\frac{1}{2} =: \Omega$. In particular, there holds $\Omega^*\Omega \neq I + D^*_dD_d$, which has been highlighted in Section \ref{sec:Intro}.
\end{remark}

A Riccati equation is given in the next theorem. There, the transfer function $\boldsymbol{\chi}$ of $S_{sp}$, and in particular its limit, given by $\Omega = P^{\frac{1}{2}}$, plays an important role.

\begin{theorem}\label{thm:Riccati}
Let $Z = \Pi I_X$ with $\Pi $ a solution of the CARE \eqref{CARE}. Moreover, let $\Omega = P^{\frac{1}{2}}$
and $B_{\omega}^*$ be the weak Yosida extension of $B^*$, defined in \eqref{Yosida_Ext_B}. Then, the following operator Riccati equation
\begin{align}
&\langle Az,Zz\rangle + \langle Zz,Az\rangle + \langle Cz,Cz\rangle\nonumber\\
&= \langle (\Omega^*\Omega)^{-1}(B_\omega^*Z + D^*C)z,(B_\omega^*Z + D^*C)z\rangle
\label{eq:Riccati}
\end{align}
holds for all $z\in D(A)$, where $A$ and $B$, see \eqref{Op(A)}--\eqref{D(A)} and \eqref{Adj_B}, stand for the operator dynamics and the control operator associated to \eqref{StateEquation_Diag_Uniform}--\eqref{Output_Diag_Uniform}, respectively, while the operators $C$ and $D$, see \eqref{OpC} and \eqref{OpD}, are the output and the feedthrough operators associated to \eqref{StateEquation_Diag_Uniform}--\eqref{Output_Diag_Uniform}, respectively.
\end{theorem}

\begin{proof}
    Consider $z\in D(A)$. Expanding the left-hand side of \eqref{eq:Riccati} yields that
    \begin{align*}
        &\langle Az,Zz\rangle + \langle Zz,Az\rangle + \langle Cz,Cz\rangle\\
        =&~ -\int_0^1 z^*\lambda_0\Pi \frac{\d}{\d\zeta}(\lambda_0 z)\d\zeta - \int_0^1\frac{\d}{\d\zeta}(\lambda_0z^*)\lambda_0\Pi z\d\zeta\\
        & + (\lambda_0z)(1)^*C_d^*C_d(\lambda_0z)(1).
    \end{align*}
    An integration by parts implies that
    \begin{align*}
        &\langle Az,Zz\rangle + \langle Zz,Az\rangle + \langle Cz,Cz\rangle\\
        =&~ (\lambda_0z)(0)^*\Pi (\lambda_0z)(0) - (\lambda_0z)(1)^*\Pi (\lambda_0z)(1)\\
        & + (\lambda_0z)(1)^*C_d^*C_d(\lambda_0z)(1).
    \end{align*}
    By using $z\in D(A)$, see \eqref{D(A)}, we get
    \begin{align}
        &\langle Az,Zz\rangle + \langle Zz,Az\rangle + \langle Cz,Cz\rangle\nonumber\\
        &= (\lambda_0z)(1)^*\left[A_d^*\Pi A_d - \Pi +C_d^*C_d\right](\lambda_0z)(1).
        \label{LHS_Riccati}
    \end{align}
    Applying the definition of $\Omega$, the boundary conditions for $z\in D(A)$, see \eqref{D(A)}, and by noting that $(B_\omega^*Z + D^*C)z = B_d^*\Pi (\lambda_0z)(0) + D_d^*C_d(\lambda_0z)(1)$, there holds
    \begin{align}
        &\langle (\Omega^*\Omega)^{-1}(B_\omega^*Z + D^*C)z,(B_\omega^*Z + D^*C)z\rangle\nonumber\\
        &=(\lambda_0z)(1)^*\left[V^*P^{-1}V\right](\lambda_0z)(1).
        \label{RHS_Riccati}
    \end{align}
    Since $\Pi $ satisfies the Riccati equation \eqref{CARE} by assumption, we get the equality between \eqref{LHS_Riccati} and \eqref{RHS_Riccati}.
\end{proof}

\section{An example highlighting the importance of \texorpdfstring{\eqref{Riccati_Unbounded}}{} for solving the LQ-optimal control problem for \texorpdfstring{\eqref{StateEquation_Diag_Uniform}--\eqref{Output_Diag_Uniform}}{}}\label{sec:CE}
In this section, we highlight the importance of considering \eqref{Riccati_Unbounded} instead of \eqref{Riccati_WeakForm} when solving the LQ-optimal optimal control for \eqref{StateEquation_Diag_Uniform}--\eqref{Output_Diag_Uniform}. Although a similar kind of approach has already been used in \cite{Weiss_Zwart}, the example presented here fits the class \eqref{StateEquation_Diag_Uniform}--\eqref{Output_Diag_Uniform}. This is not the case of the example given in \cite{Weiss_Zwart}, in which the observation operator is bounded. We consider the following transport equation
\begin{equation}
\frac{\partial z}{\partial t}(\zeta,t) = -\frac{\partial z}{\partial \zeta}(\zeta,t),\,\, z(\zeta,0) = z_0(\zeta)
\label{State_CE}
\end{equation}
together with the boundary inputs and outputs 
\begin{align}
u(t) &= z(0,t) + \frac{1}{2}z(1,t),\label{Input_CE}\\
y(t) &= z(0,t).\label{Output_CE}
\end{align}
According to Section \ref{sec:SystemClass}, there holds
\begin{align*}
K &= -1,~ L = -\frac{1}{2},~ K_y = -1,~ L_y = 0,
\end{align*}
respectively. This implies that the matrices $A_d, B_d, C_d$ and $D_d$ are expressed as
\begin{align*}
A_d = -\frac{1}{2},~ B_d = 1,~ C_d = -\frac{1}{2},~ D_d = 1,
\end{align*}
respectively. Moreover, \eqref{State_CE} with  \eqref{Input_CE} and \eqref{State_CE} may be written in an abstract way as $\dot{z}(t) = Az(t) + Bu(t)$ with $y(t) = Cz(t) + Du(t)$, where the (unbounded) operator $A$ is defined as 
\begin{align*}
Az &= -\frac{\d z}{\d\zeta},\\
D(A) &= \{z\in \H^1(0,1;\mathbb{C}), z(0) + \frac{1}{2}z(1) = 0\}.
\end{align*}
In addition, according to \eqref{Adj_B}, the adjoint of $B$, denoted $B^*$, is given by $B^*: D(A^*)\to \mathbb{C}, B^*z = z(0)$, where 
\begin{align*}
A^*z &= \frac{\d z}{\d\zeta},\\
D(A^*) &= \{z\in\H^1(0,1;\mathbb{C}), z(1) + \frac{1}{2}z(0) = 0\}.
\end{align*} 
The operators $C:D(A)\to\mathbb{C}$ and $D:\mathbb{C}\to\mathbb{C}$ are defined by $Cz = -\frac{1}{2}z(1)$ and $Du = u$, see \eqref{OpC} and \eqref{OpD}, respectively. According to \eqref{eq:OptimalFeedback}, the optimal control input that minimizes \eqref{Cost} is given by
\begin{align*}
u^{opt}(t) = \frac{1}{2}(2 + \Pi)^{-1}(1 + \Pi)z^{opt}(1,t),
\end{align*}
where $\Pi$ is the nonnegative solution to the following scalar Riccati equation
\begin{equation}
4\Pi^2 + 7\Pi - 1 = 0,
\label{Scalar_Riccati}
\end{equation}
see \eqref{CARE}. The optimal cost is given by \eqref{OptimalCost}. Let us focus now on the Riccati equation \eqref{Riccati_WeakForm}. According to \eqref{OptimalCost}, a natural choice for the operator $\Theta$ is given by $\Theta z = \Pi z$, with $\Pi$ a solution to \eqref{Scalar_Riccati}. From this definition and the definition of $B^*$, it is clear that for $z\in D(A)$, $\Theta z$ does not belong to $D(A^*)$, which means that the operator $B^*$ may not be applied to $\Theta z$. To solve this, we replace $B^*$ by its Yosida extension, given by $B^*_\omega f = f(0^+)$, see \eqref{Yosida_Ext_B}. The problem highlighted before is then solved by replacing $B^*$ by $B_\omega^*$ in \eqref{Riccati_WeakForm}. Unfortunately, another problem appears when considering \eqref{Riccati_WeakForm}. To see this, let us compute the left-hand side of \eqref{Riccati_WeakForm}. For $z\in D(A)$, there holds
\begin{align*}
&\langle Az,\Theta z\rangle + \langle\Theta z,A z\rangle + \langle Cz,Cz\rangle\\
&=-\Pi\int_0^1 z^*\frac{\d z}{\d\zeta}\d\zeta - \Pi \int_0^1\frac{\d z^*}{\d\zeta}z\d\zeta + \frac{1}{4}z(1)^*z(1).
\end{align*}
Integration by parts and $z\in D(A)$ imply that
\begin{align}
\langle Az,\Theta z\rangle + \langle\Theta z,A z\rangle + \langle Cz,Cz\rangle &= \Pi z(0)^*z(0) + \frac{1-4\Pi}{4}z(1)^*z(1)\nonumber\\
&= \frac{1-3\Pi}{4}z(1)^*z(1).\label{LHS_CE}
\end{align}
In addition, the right-hand side of \eqref{Riccati_WeakForm} with $B^*$ replaced by $B^*_\omega$ is given by 
\begin{align*}
&\langle (I + D^*D)^{-1}(B^*_\omega\Theta + D^*C)z,(B^*_\omega\Theta + D^*C)z\rangle\\
&= (I + D^*D)^{-1}\left(\Pi z(0)^* - \frac{1}{2}z(1)^*\right)\left(\Pi z(0) - \frac{1}{2}z(1)\right).
\end{align*}
Using the equality $z(0) = -\frac{1}{2}z(1)$ for $z\in D(A)$ and the definition of $D$ implies that
\begin{align}
&\langle (I + D^*D)^{-1}(B^*_\omega\Theta + D^*C)z,(B^*_\omega\Theta + D^*C)z\rangle\nonumber\\
&=\frac{(\Pi + 1)^2}{8}z(1)^*z(1).\label{RHS_CE}
\end{align}
Equality between \eqref{LHS_CE} and \eqref{RHS_CE} implies $\Pi^2 + 8\Pi - 1 = 0$, which is different from \eqref{Scalar_Riccati}. Hence, the Riccati equation \eqref{Riccati_WeakForm} has to be adapted. In particular, the operator $(I + D^*D)^{-1}$ in the right-hand side has to be modified as explained in Section \ref{sec:Riccati_SpecFact}. In particular, $(I + D^*D)^{-1}$ has to be replaced by $\Omega^*\Omega$ with $\Omega = P^{\frac{1}{2}}, P = I + D_d^*D_d + B_d^*\Pi B_d$, see Remark \ref{rem:SpecFact}. Making this replacement in \eqref{RHS_CE} implies that
\begin{align}
&\langle (I + D_d^*D_d + B_d^*\Pi B_d)^{-1}(B^*_\omega\Theta + D^*C)z,(B^*_\omega\Theta + D^*C)z\rangle\nonumber\\
=&\frac{(\Pi + 1)^2}{4(\Pi+2)}.\label{RHS_CE_Modif}
\end{align}
Equality between \eqref{LHS_CE} and \eqref{RHS_CE_Modif} is equivalent to $4\Pi^2 + 7\Pi - 1 = 0$, which is identical to \eqref{Scalar_Riccati}.

\section{Perspectives}\label{sec:Ccl}
Extending the proposed approach to a more general class of PDEs, allowing for instance to replace the function $\lambda_0$ by a diagonal matrix with possibly different entries could be considered as future work. In addition, the dual versions of the Riccati equations \eqref{Riccati_Unbounded} and \eqref{ON_Riccati} could be considered in order to tackle the Kalman filter problem. The meaning of duality in this context should be properly addressed because of the unboundedness of both the control and the observation operators.

\section*{Acknowledgments}
This work was supported by the German Research Foundation (DFG). A. H. is supported by the DFG under the Grant HA 10262/2-1.

\bibliographystyle{elsarticle-num} 
\bibliography{biblio_Main}

\end{document}